\documentclass[a4paper, 11pt]{amsart}

\usepackage{mathpazo}

\usepackage{amsrefs}
\usepackage{amsfonts}
\usepackage{amssymb}
\usepackage{amsthm}
\usepackage{amsmath}
\usepackage[all]{xy}

\DeclareMathOperator{\Hom}{Hom} 
\DeclareMathOperator{\Ext}{Ext} 
\DeclareMathOperator{\Tor}{Tor} 
\DeclareMathOperator{\img}{im} 
 \DeclareMathOperator{\E}{E}
\DeclareMathOperator{\D}{D} 
\DeclareMathOperator{\Tr}{T}
\newcommand{\id}{\mathrm{id}}
\newcommand{\ktwo}{\mathcal{K}_2}

\newcommand{\T}{\mathbb{T}}

\newcommand{\bbk}{\Bbbk}

\newtheorem{theorem}[subsection]{Theorem}
\newtheorem{lemma}[subsection]{Lemma}
\newtheorem{corollary}[subsection]{Corollary}
\theoremstyle{definition}
\newtheorem{definition}[subsection]{Definition}
\newtheorem{example}[subsection]{Example}
\theoremstyle{remark}

\begin{document}

\title{The Yoneda algebra of a graded Ore extension}

\author{Christopher Phan}

\address{Department of Mathematics\\University of Glasgow\\Glasgow
G12 8QW\\United Kingdom}

\email{c.phan@maths.gla.ac.uk}

\urladdr{http://www.maths.gla.ac.uk/~clphan/}



\begin{abstract}
Let $A$ be a connected-graded algebra with trivial module $\bbk$,
and let $B$ be a graded Ore extension of $A$.  We relate the
structure of the Yoneda algebra $\E(A) := \Ext_A(\bbk, \bbk)$ to
$\E(B)$. Cassidy and Shelton have shown that when $A$ satisfies
their $\ktwo$ property, $B$ will also be $\ktwo$. We prove the
converse of this result.
\end{abstract}

\maketitle

\section{Introduction}

Let $A$ be a connected-graded algebra over a field $\bbk$ generated
in degree $1$, and let $_A\bbk$ be the $A$-module $A/A_+$. We then
have a bigraded algebra $\E(A) := \Ext_A(\bbk, \bbk)$, called the
{\bf Yoneda algebra} of $A$. (Throughout, $\Ext$ and $\Tor$ refer to
the functors on the \emph{graded} category. We define $\E^i(A) =
\bigoplus_j \E^{i,j}(A)$, where the index $i$ refers to the
homological grading and $j$ refers to the grading inherited from
$A$.) Frequently studied is a property introduced by Priddy in
\cite{priddy70}: an algebra $A$ is {\bf Koszul} if $\E(A)$ is
generated in the first cohomological degree (equivalently,
$\E^{i,j}(A) = 0$ when $i \not = j$).  For a comprehensive reference
on Koszul algebras, see \cite{pp}.

In \cite{csk2} Cassidy and Shelton proposed the following
generalization of Koszul:
\begin{definition} A connected-graded algebra $A$ generated in degree $1$ is a $\ktwo$
algebra if $\E(A)$ is generated in the first two cohomological
degrees.
\end{definition} \noindent This class includes the class of
$N$-Koszul algebras defined by Berger in \cite{berger}. For
quadratic algebras, $\ktwo$ is equivalent to Koszul. The study of
$\ktwo$ algebras is attractive because the first two cohomological
degrees of $\E(A)$ are seen in a suitable presentation of $A$: we
have $\E^1(A) \simeq A_1^*$ while $\E^2(A)$ is roughly dual to the
essential relations of $A$. The class of $\ktwo$ algebras has been
further studied in \cite{cpsk2} and \cite{phan09}.

In addition to exploring the connections between $A$ and $\E(A)$,
algebraists have investigated structural similarities between
$\E(A)$ and $\E(B)$ when $B$ is an algebra related to $A$. For
example (see \cite{pp,csk2}), both the classes of Koszul and $\ktwo$
algebras are closed under regular central extensions, Zhang twists,
and tensor products.

Furthermore, both the class of Koszul algebras and the class of
$\ktwo$ algebras are closed under graded Ore extensions. Suppose
$\sigma:A \rightarrow A$ is a graded algebra automorphism and
$\delta: A(-1) \rightarrow A$ is a graded $\sigma$-derivation. Then
the associated Ore extension $B := A[z; \sigma, \delta]$ is also a
connected-graded algebra. (A primer on Ore extensions can be found
in \cite{bluebook}.) In \cite{csk2}, it was proved that if $A$ is a
$\ktwo$ algebra, then $B$ is also $\ktwo$. This is a generalization
of the well-known result that such a graded Ore extension of a
Koszul algebra is again Koszul.

Our goal in this article is to provide another connection between
$\E(A)$ and $\E(B)$ when $B$ is such a graded Ore extension of $A$:
\begin{theorem}\label{t:main} $B$ is $\ktwo$ only if $A$ is $\ktwo$.\end{theorem}

Note that $B$ is quadratic if and only if $A$ is quadratic. We
therefore obtain this corollary:
\begin{corollary} $B$ is Koszul only if $A$ is
Koszul.\end{corollary}

In Section \ref{s:resoln} of this article, we describe a procedure
that produces a projective resolution of $_B\bbk$ based on a minimal
projective resolution of $_A\bbk$. In Section \ref{s:maps}, we
create some maps that relate the structure of $\E(A)$ and $\E(B)$.
In Section \ref{s:proof}, we prove Theorem \ref{t:main}. We conclude
in Section \ref{s:examples} with some examples.

The author wishes to thank Kenneth Brown, Thomas Cassidy, Ulrich
Kr\"ahmer, and Brad Shelton for their helpful comments and
suggestions.

\section{Projective resolution for $_B\bbk$} \label{s:resoln}

A projective resolution
\[ \cdots \rightarrow P_n \rightarrow P_{n-1}
\rightarrow \cdots \rightarrow P_0\] of the $A$-module $M$ is {\bf
minimal} if the differentials in the sequence
\[ \Hom(P_0, \bbk) \rightarrow \cdots \rightarrow \Hom_A(P_{n-1}, \bbk) \rightarrow \Hom_A(P_n,
\bbk) \rightarrow \cdots\] are all zero. In this section, we
describe a procedure for constructing a projective resolution of the
$_B\bbk$ based on a minimal projective resolution of $_A\bbk$.
However, the resulting resolution for $_B\bbk$ will not necessarily
be minimal.

We begin by setting some notation in effect for the rest of the
article. Let $V$ be a finite-dimensional vector space over a field
$\bbk$, and let $\T(V)$ denote the tensor algebra. Let $I \subseteq
\sum_{n \geq 2} V^{\otimes n}$ be a finitely-generated homogeneous
ideal of $\T(V)$ and let $A := \T(V)/I$, which is a connected-graded
$\bbk$-algebra generated in degree $1$. Let $\sigma: A \rightarrow
A$ be a graded algebra automorphism, and $\delta: A(-1) \rightarrow
A$ be a graded $\sigma$-derivation---that is, for $a_1, a_2 \in A$,
we have
\[\delta(a_1 a_2) = \delta(a_1) \sigma(a_2) + a_1 \delta(a_2).\] Let
$B := A[z; \sigma, \delta]$ be the associated Ore extension---that
is, $B = \bigoplus_{n \geq 0} z^nA$ as an $A$-module, and for $a \in
A$, \[az = z\sigma(a) + \delta(a).\] We consider $z$ to have degree
$1$ in $B$, and under this grading $B$ is a connected-graded algebra
generated in degree $1$. In fact, $B := \T(W)/J$ for some ideal $J$,
where $W = V \oplus \bbk z$.

The following construction is a graded version of a construction by
Gopalakrishnan and Sridharan in \cite{gs66}. We recreate it from
scratch because we will refer to the details of the construction
later.

Let $V_n \subseteq A_+^{\otimes n}$ such that $A \otimes V_\bullet$,
as a subcomplex of the bar resolution $A \otimes A_+^{\otimes
\bullet}$, is a minimal projective resolution of the trivial module
$_A\bbk$. We may assume $V_0 = \bbk$ and $V_1 = V$.

\begin{definition}
We construct two chain complexes of projective $B$-modules:
\begin{enumerate}
\item $P_\bullet := B \otimes_A (A \otimes V_\bullet)(-1)$, which is
canonically isomorphic to $B \otimes V_\bullet(-1)$. (Here, the
degree shift is by internal degree.)
\item $Q_\bullet$ is the complex with
\[Q_i  = \begin{cases}
 B \otimes V_i, & \text{for $i \geq 0$,}\\
 \bbk, & \text{for $i = -1$},
 \end{cases}\]
 and differential $\partial_{i,Q} =
  -\partial_{i, P}$ for $i \geq 1$, and $\partial_{0, Q}: B \otimes
  V^0 \rightarrow \bbk$ given by the augmentation.
\end{enumerate}
\end{definition}

We begin by computing the homology of the complexes $P_\bullet$ and
$Q_\bullet$.

\begin{lemma} \label{l:homologycalc} We have
\[H_i(P_\bullet) = \begin{cases}
\bigoplus_{n \geq 0} \bbk( z^n + BA_+)(-1) , & \text{if $i = 0$},\\
0,& \text{otherwise,}
\end{cases}\]
and
\[H_i(Q_\bullet) = \begin{cases}
\bigoplus_{n \geq 1} \bbk( z^n + BA_+), & \text{if $i = 0$},\\
0, & \text{otherwise},
\end{cases}
\] where $z^n + BA_+ \in B/(BA_+)$.
\end{lemma}

\begin{proof}
Since $H_i(P_\bullet) = \Tor^A_i(B, \bbk)(-1)$ and $B$ is a
graded-free $A$-module, the first statement holds. Furthermore,
$H_i(Q_\bullet) = H_i(P_\bullet)$ when $i \geq 1$. The statement
about $H_0(Q_\bullet)$ and $H_{-1}(Q_\bullet)$ can be verified by
direct calculation.
\end{proof}

\begin{lemma}[c.f. {\cite[Lemma of Theorem 1]{gs66}}] \label{l:fexists}
There exists a graded chain map $f: P_\bullet \rightarrow Q_\bullet$
with $f_0(b \otimes \lambda) := bz \otimes \lambda$, $f_{-1} := 0$,
$\img f_0 \subseteq B_+ \otimes V_0$, and $\img f_1 \subseteq B_+
\otimes V_1$.
\end{lemma}

\begin{proof}
By construction, $\img f_0 \subseteq B_+ \otimes V_0$. We may lift
$\delta: A_1 = V_1 \rightarrow A_2$ to a map $\tilde{\delta}: V_1
\rightarrow A_1 \otimes V_1 \subseteq B_+ \otimes V_1$. Then for $x
\in V_1$, put \[f_1(1 \otimes x) := z \otimes \sigma^{-1}(x) -
\tilde{\delta}(x) \in B_+ \otimes V_1.\] Computation shows that
$\partial_Q \circ f_1 = \partial_P \circ f_0$.

 By Lemma
\ref{l:homologycalc}, $Q_\bullet$ is exact at every term except at
$Q_0$. Therefore, the construction of the rest of the chain map
follows automatically.
\end{proof}

We can now construct a projective resolution of $_B\bbk$ from the
resolution $A \otimes V_\bullet$.

\begin{lemma}[c.f. {\cite[Theorem 1]{gs66}}] The algebraic mapping cone of $f$,
\begin{multline*}
 \cdots \rightarrow B \otimes (V_n \oplus V_{n-1}(-1))
\xrightarrow{\left( \begin{matrix}
\partial_{n} & 0 \\
f_n & \partial_{n-1} \end{matrix} \right)} B \otimes (V_{n-1} \oplus
V_{n-2}(-1)) \rightarrow \cdots\\
\rightarrow B \otimes (V_1 \oplus V_0(-1)) \xrightarrow{\left(
\begin{matrix}
\partial_{1} \\
f_0 \end{matrix} \right) } B \rightarrow \bbk \rightarrow 0,
\end{multline*}
is exact. (Here, the functions in each matrix act on the right.)
Hence, if we write $\mathrm{cone}(f)$ in the form $B \otimes
W_\bullet \rightarrow \bbk \rightarrow 0$, then $B \otimes
W_\bullet$ is a projective resolution of $_B\bbk$.
\end{lemma}

\begin{proof}
We consider $f_\bullet: H_\bullet(P_\bullet) \rightarrow
H_\bullet(Q_\bullet)$, the induced map on homology. Direct
calculation shows $f_0: H_0(P_\bullet) \rightarrow H_0(Q_\bullet)$
is an isomorphism. By Lemma \ref{l:homologycalc}, $f_n:
H_n(P_\bullet) \rightarrow H_n(Q_\bullet)$ is an isomorphism for $n
\not = 0$. Therefore, the chain map $f$ is a quasi-isomorphism, and
the algebraic mapping cone is exact (see, for example,
\cite[Corollary 1.5.4]{weibel}).
\end{proof}

It is \emph{not} true that $B \otimes W_\bullet \rightarrow \bbk
\rightarrow 0$  must be a \emph{minimal} projective resolution. For
the resolution to be minimal, $f_n(B \otimes V_n(-1))$ must lie
entirely inside $ B_+ \otimes V_n$ (or, equivalently, each entry of
the matrix representation of each $f_n$ must be an element of
$B_+$). This is not the case in Example \ref{e:fwithdegzero}.

\section{Maps between $\E(A)$ and $\E(B)$} \label{s:maps}
In this section, we consider maps relating $\E(A)$ and $\E(B)$.
Because $B \otimes W_n = B \otimes V_n \oplus B \otimes
V_{n-1}(-1)$, it is tempting to view $\E^n(B)$ as $\E^n(A) \oplus
\E^{n-1}(A)(-1)$. However, there are many obstructions to this,
including the fact that $\E(A)$ is not a subalgebra of $\E(B)$ and
that the resolution $B \otimes W_\bullet$ is not minimal.

However, enough of $B \otimes W_\bullet$ is minimal to prove the
following:
\begin{lemma} \label{l:iotasurj}
$\E(\iota)(\E^1(B)) = \E^1(A)$ and $\E(\iota)(\E^2(B)) = \E^2(A)$,
where $\iota: A \hookrightarrow B$ is the inclusion.
\end{lemma}

\begin{proof}
Since $\img f_0 \subseteq B_+ \otimes V_0$ and $\img f_1 \subseteq
B_+ \otimes V_1$, the sequence
\[
 B \otimes (V_2 \oplus V_{1}(-1))
\xrightarrow{\left( \begin{matrix}
\partial_{2} & 0 \\
f_1 & \partial_{1} \end{matrix} \right)} B \otimes (V_1 \oplus
V_0(-1)) \xrightarrow{\left(
\begin{matrix}
\partial_{1} \\
f_0 \end{matrix} \right) } B
\]
may be extended to a minimal projective resolution $B \otimes
U_\bullet$ of $_B\bbk$. Now, the inclusion $\iota: A \hookrightarrow
B$ induces a chain map $\tilde{\iota}: A \otimes V_\bullet
\rightarrow B \otimes U_\bullet$ where $\tilde{\iota}_1 (a \otimes
x) = \iota(a) \otimes (x, 0)$  and $\tilde{\iota}_2 (a \otimes x) =
\iota(a) \otimes (x, 0)$.\end{proof}

We now define some chain maps for use later. Let
\[\tilde{\varphi}_n: \Hom_\bbk(V_{n-1}(-1), \bbk) \rightarrow
\Hom_\bbk(V_n \oplus V_{n-1}(-1), \bbk)\] be the map dual to the
projection $V_n \oplus V_{n-1}(-1) \twoheadrightarrow V_{n-1}(-1)$
onto the second direct summand. Let \[\tilde{\tau}_n: \Hom_\bbk(V_n
\oplus V_{n-1}(-1), \bbk) \rightarrow \Hom_\bbk(V_n, \bbk)\] be the
restriction to the first direct summand.
\begin{lemma} \label{l:varphiandtau}
\begin{enumerate}
\item $\tilde{\varphi}_n$ and $\tilde{\tau}_n$ induce maps $\varphi_n:
\E^{n-1}(A)(1) \rightarrow \E^{n}(B)$ and $\tau_n:\E^{n}(B)
\rightarrow \E^n(A)$, respectively.
\item $\tau_n = \E^n(\iota)$, where $\iota:A
\hookrightarrow B$ is the inclusion.
\item $\varphi_n(\E^{n-1}(A)) = \ker(\tau_n)$.
\end{enumerate}
\end{lemma}

\begin{proof} Recall that we have natural isomorphisms
\[\Hom_A(A \otimes -, \bbk) \simeq \Hom_\bbk(-, \bbk) \text{ and } \Hom_B(B
\otimes -, \bbk) \simeq \Hom_\bbk(-, \bbk).\] The map
$\tilde{\varphi}_n$ induces a map $\varphi_n: \E^{n-1}(A)
\rightarrow \E^{n}(B)$ because the diagram \[\xymatrix{ V_n \oplus
V_{n-1}(-1) \ar[rr]^{\left(\begin{matrix} 0 \cr 1 \end{matrix}
\right)} \ar[d]^{\left(\begin{matrix} 0 & 0 \cr f'_{n-1} &
0\end{matrix}\right)} & & V_{n-1} \ar[d]^{\begin{matrix} 0
\end{matrix}} \cr V_{n-1} \oplus V_{n-2}(-1)
\ar[rr]_{\left(\begin{matrix} 0 \cr 1
\end{matrix}\right)} & & V_{n-2} }\]
commutes, where $f_{n-1}'$ is the composition
\[V_{n-1}(-1) \hookrightarrow B \otimes V_{n-1}(-1) \xrightarrow{f_{n-1}}
B \otimes V_{n-1} \twoheadrightarrow B/B_+ \otimes V_{n-1} \simeq
V_{n-1}.\] We have $\tau_n = \E^n(\iota)$ because $\tilde{\tau}_n$
is dual to the chain map
\[A \otimes V_n \xrightarrow{\iota \otimes \id} B \otimes V_n
\hookrightarrow B \otimes (V_n \oplus V_{n-1}(-1)).\]

It is clear that $\varphi_n(\E(A)) \subseteq \ker(\tau_n)$. Now,
suppose $\xi \in \Hom_\bbk(V_n \oplus V_{n-1}(-1), \bbk)$ represents
a cohomology class $[\xi] \in \ker(\tau_n) \subseteq \E^n(B)$. As $A
\otimes V_\bullet$ is a minimal projective resolution for $_A\bbk$,
this implies that $\tilde{\tau}_n(\xi) = 0$, that is,  $\xi|_{V_n
\oplus 0} = 0$. So, there exists $\xi' \in \Hom_\bbk(V_{n-1}(-1),
\bbk)$ such that $\xi = (0, \xi')$, that is, $\tilde{\varphi}(\xi')
= \xi$.
\end{proof}

Let $\overline{z} \in W^*$ such that $\overline{z}(z) = 1$ and
$\overline{z}(V) = 0$. We also use the notation $\overline{z}$ to
mean the induced element of bidegree $(1,1)$ in the Yoneda algebra
$\E(B)$.  Let $\D(A)$ be the subalgebra of $\E(A)$ generated as a
$\bbk$-algebra by $\E^1(A)$ and $\E^2(A)$. We define $\D(B)$
analogously.

In the case where $A$ is $\ktwo$, \cite{csk2} exhibits short exact
sequences \begin{equation} \label{e:csses} 0 \rightarrow
\overline{z} \E^{n-1}(B) \rightarrow \E^n(B) \rightarrow \E^n(A)
\rightarrow 0.\end{equation} Our goal is to show that if $B$ is
$\ktwo$, there are similar short exact sequences
\[0 \rightarrow \overline{z} \E^{n-1}(B) \rightarrow \E^n(B)
\rightarrow \D^n(A) \rightarrow 0 .\]

We recall some more results from \cite{csk2}. As before, $\iota:A
\hookrightarrow B$ is the inclusion.
 Noting that $A_+B = BA_+$ is an ideal of $B$, set $C := B/A_+B$. As a $\bbk$-algebra, $C
\simeq \bbk[z]$, and has a $B$-module endomorphism $\zeta: C
\rightarrow C$ via right-multiplication by $z$. Furthermore, $_B C =
B \otimes_A \bbk$, meaning $\Ext_B(C, \bbk) = \E(A)$. Hence, the
short exact sequence
\[0 \rightarrow C(-1) \xrightarrow{\zeta} C \rightarrow \bbk
\rightarrow 0\] yields a long exact sequence
\begin{equation}
\label{e:csles} \cdots \rightarrow \E^{n-1}(A)(1)
\xrightarrow{\alpha} \E^n(B) \xrightarrow{\E(\iota)} \E^n(A)
\xrightarrow{\zeta^*} \E^n(A)(1) \rightarrow \cdots. \end{equation}
(It is this sequence that breaks into the short exact sequences
(\ref{e:csses}) above when $A$ is $\ktwo$.) The map $\zeta^*: \E(A)
\rightarrow \E(A)(0,1)$ is a $\E(\iota)$-derivation, and vanishes on
$\E^1(A)$ and $\E^2(A)$.

\begin{lemma} $\E(\iota)(\D(B)) = D(A)$. \end{lemma}

\begin{proof}
By Lemma \ref{l:iotasurj}, $\E(\iota)(\D^1(B)) = \E^1(A) = \D^1(A)$
and $\E(\iota)(\D^2(B)) = \E^2(A) = \D^2(A)$.
\end{proof}

\begin{lemma} \label{l:ses} Suppose $B$ is $\ktwo$. Then the long exact sequence
(\ref{e:csles}) breaks into short exact sequences
\begin{equation}\label{e:myses} 0 \rightarrow \overline{z} \E^{n-1}(B)
\rightarrow \E^n(B) \rightarrow \D^n(A) \rightarrow 0 \hbox{\hskip 1
cm} (n \geq 1).\end{equation}
\end{lemma}

\begin{proof}
As $\zeta^*$ is a $\E(\iota)$-derivation which vanishes on $\E^1(A)$
and $\E^2(A)$, and $\D(A)$ is generated by $\E^1(A)$ and $\E^2(A)$,
$\zeta^*$ vanishes on all of $\D(A)$. Thus, $\alpha$ in injective on
$\D(A)(1)$. Therefore, for each $n$, we have short exact sequences
\[0 \rightarrow \D^{n-1}(A)(1) \xrightarrow{\alpha} \E^n(B)
\xrightarrow{\E(\iota)} \D^n(A) \rightarrow 0.\]

However, since $\ker(\E(\iota): \E^1(B) \rightarrow \D^1(A)) = \bbk
\overline{z}$, surjectivity of $\E(\iota)$ implies that
\[\alpha(\D^{n-1}(A)(1)) = \overline{z} \E^{n-1}(B).\qedhere\]
\end{proof}

It is no coincidence that the previous proof is very similar to the
proof of \cite[Theorem 10.2]{csk2}, in which the short exact
sequence (\ref{e:csses}) appeared. Later, we will show that $B$
being $\ktwo$ will imply that $A$ is also $\ktwo$, meaning that the
short exact sequences (\ref{e:csses}) and (\ref{e:myses}) are indeed
the same.

We now relate the maps $\alpha$ and $\varphi$.
\begin{lemma} \label{l:alphavarphi}
Suppose $B$ is $\ktwo$, and $\E^{n,m}(A) = \D^{n,m}(A)$. Then
$\varphi_n$ is injective on $\E^{n,m}(A)(1)$.
\end{lemma}

\begin{proof} By Lemmas \ref{l:varphiandtau} and \ref{l:ses}, $\varphi(\E^{n,m}(A))
= \overline{z} \E^{n,m}(B)$. On the other hand, \[\dim \E^{n,m}(A) =
\dim \D^{n,m}(A) = \dim \overline{z} \E^{n,m}(B),\] by injectivity
of $\alpha$ on $\D^{n}(A)$.
\end{proof}

\section{Proof of the main theorem} \label{s:proof}
Again, we continue to use the notation established in the previous
section. We are now ready to prove Theorem \ref{t:main}, which we
restate.
\begin{theorem} If $B$ is $\ktwo$, then $A$ is also $\ktwo$.
\end{theorem}

\begin{proof}
Suppose $A$ is not $\ktwo$. Then there is a unique bidegree $(n,m)$
such that $\E^{n,m}(A) \not = \D^{n,m}(A)$ but $\E^{i,j}(A) =
\D^{i,j}(A)$ when $i < n$ or when $i = n$ and $j < m$. Let $\xi:
V_{n,m} \rightarrow \bbk$ represent a cohomology class $[\xi] \in
\E^{n,m}(A) \setminus \D^{n,m}(A)$. Under the natural isomorphism
\[\Hom_\bbk(V_n, \bbk) \simeq \Hom_A(A \otimes V_n, \bbk),\] we can
also view $\xi: A \otimes V_n \rightarrow \bbk$.

Recall the chain map $f$ and the projective resolution $B \otimes
W_\bullet \rightarrow \bbk \rightarrow 0$ in the previous section.
Let $\xi'$ be the map $(\xi, 0): B \otimes (V_n \oplus V_{n-1}(-1))
\rightarrow \bbk$.

We wish to show $\xi'$ represents a cohomology class $[\xi']$ in
$\E^n(B)$. Under the natural isomorphism \[\Hom_\bbk(V_n \oplus
V_{n-1}, \bbk) \simeq \Hom_B(B \otimes (V_n \oplus V_{n-1}),
\bbk),\] we have \[\partial^*_{B \otimes W_\bullet}(\xi') =
(\partial_n^*(\xi),f_n^*(\xi)): V_{n+1, m} \oplus V_n(-1)_m
\rightarrow \bbk.\] By degree considerations, $\partial_n^*(\xi) =
0$, and therefore we need only consider $f_n^*(\xi): V_n(-1)_m
\rightarrow \bbk$.

Note that $V_n(-1)_m = V_{n,m-1}$. By choice of $n$ and $m$,
$E^{n,m-1}(A) = \D^{n,m-1}(A)$, and so by Lemma \ref{l:alphavarphi},
$\varphi: \E^{n, m-1}(A) \rightarrow \E^{n+1, m}(B)$ is injective.
Therefore, any element of \[\Hom_\bbk(V_n(-1)_m, \bbk) \subset
\Hom_\bbk(V_{n+1, m} \oplus V_n(-1)_m, \bbk)\] represents a nonzero
cohomology class in $\E^{n+1, m}(B)$.  Therefore,
$\Hom_\bbk(V_n(-1)_m, \bbk) \cap \partial^*_{B \otimes W_\bullet} =
0$. Hence, $\partial^*_{B \otimes W_\bullet}(\xi') = 0$, and
$[\xi']$ is a cohomology class in $\E^{n}(B)$.

However, by Lemma \ref{l:varphiandtau}, $\E(\iota)([\xi']) =
[\tilde\tau(\xi')] = [\xi]$, contrary to $\E(\iota)(\E(B)) = \D(A)$.
\end{proof}

\section{Examples} \label{s:examples}

We conclude with some examples. Throughout, $A(d_1, d_2, d_3,
\dots)$ denotes the graded-free $A$-module $A(d_1) \oplus A(d_2)
\oplus A(d_3) \oplus \cdots$. Let $\Tr(A) := \Tor^A(\bbk, \bbk)$,
which is a bigraded coalgebra, with comultiplications $\Tr_{i+j}(A)
\rightarrow \Tr_i(A) \otimes \Tr_j(A)$ induced by $A_+^{\otimes i+j}
\rightarrow A_+^{\otimes i} \otimes A_+^{\otimes j}$. An algebra $A$
is $\ktwo$ if and only if
\[\Delta: \Tr^n(A) \rightarrow \Tr^2(A) \otimes \Tr^{n-2}(A) \oplus
\Tr^1(A) \otimes \Tr^{n-1}(A)\] is injective.

\begin{example}
This example illustrates Theorem \ref{t:main} by showing how some of
the cohomology classes in $\E(A) \setminus \D(A)$ also appear in
$\E(B)$. Let $A := \bbk\left<x, y \right> / \left<x^2 y , y^2 x
\right>$. Define an automorphism $\sigma: A \rightarrow A$ via
$\sigma(x) := y$ and $\sigma(y) := x$. Define a $\sigma$-derivation
$\delta: A(-1) \rightarrow A$ via $\delta(x) = xy$ and $\delta(y) =
yx$. Let $B := A[z; \sigma, \delta]$ be the associated Ore
extension.

A minimal projective resolution of $_A\bbk$ begins \[ \dots
\rightarrow A(-7)^2 \xrightarrow{M_4} A(-5)^2 \xrightarrow{M_3}
A(-3)^2 \xrightarrow{M_2} A(-1)^2 \xrightarrow{M_1} A \rightarrow
\bbk \rightarrow 0,
\]
where \[ M_4 := \left(
\begin{matrix}
y^2 & 0 \\ 0 & x^2 \end{matrix} \right), M_3 := \left(
\begin{matrix} x^2 & 0 \\ 0 & y^2 \end{matrix} \right), M_2 :=
 \left(
\begin{matrix}
y^2 & 0 \\ 0 & x^2
\end{matrix}
\right), \text{ and } M_1 := \left( \begin{matrix} x \\ y
\end{matrix} \right). \]

The algebra $A$ is not $\ktwo$ because $y^2 \otimes x^2 \otimes y$
represents a nonzero homology class in $\Tr_3(A)$, but
\begin{multline*}
\Delta(y^2 \otimes x^2 \otimes y + \img
\partial) = \\ (y^2 \otimes x^2 + \img \partial) \otimes (y + \img
\partial) + (y^2 + \img \partial) \otimes (x^2 \otimes y + \img
\partial) = 0.\end{multline*}

Now, we can begin a minimal projective resolution of $_B\bbk$ with
\begin{multline*}
\dots \rightarrow B(-7)^2 \oplus B(-6)^2 \xrightarrow{\left(
\begin{matrix} M_4 & 0 \\ f_3 & M_3 \end{matrix} \right)}
B(-5)^2 \oplus B(-4)^2  \xrightarrow{\left(\begin{matrix} M_3 & 0 \\ f_2 & M_2 \end{matrix}\right)}\\
B(-3)^2 \oplus B(-2)^2\xrightarrow{ \left( \begin{matrix} M_2 & 0
\\ f_1 & M_1 \end{matrix} \right)} B(-1)^2 \oplus B(-1)
\xrightarrow{\left( \begin{matrix} M_1 \\ f_0 \end{matrix} \right)}
B \rightarrow \bbk \rightarrow 0,
\end{multline*}
where \begin{multline*} f_3 :=  \left( \begin{matrix} 0 & -z - x \\
-z-y & 0
\end{matrix}\right)
f_2 := \left( \begin{matrix} 0 & -z-y \\ -z - x & 0 \end{matrix} \right),\\
f_1 := \left( \begin{matrix} 0 & -z-x \\ -z - y & 0 \end{matrix}
\right), \text{ and } f_0:= (z).
\end{multline*}
Again, $y^2 \otimes x^2 \otimes y$ represents a nonzero homology
class in $\Tr_3(B)$, and $\Delta(y^2 \otimes x^2 \otimes y + \img
\partial) = 0$, meaning $B$ is not $\ktwo$.
\end{example}

\begin{example} \label{e:fwithdegzero}
This example shows that the resolution $B \otimes W_\bullet$ need
not be a minimal projective resolution. Let $A := \bbk\left<w, x,
y,u\right>/\left<yu, ux - xu, uw\right>$. The following is a minimal
projective resolution for $_A\bbk$: \[ 0 \rightarrow A(-3, -4, -5,
\dots) \xrightarrow{M_3} A(-2)^3 \xrightarrow{M_2} A(-1)^4
\xrightarrow{M_1} A \rightarrow \bbk \rightarrow 0,\] where
\[M_3 := \left(\begin{matrix} y & 0 & 0 \cr yx & 0 & 0 \cr yx^2 & 0
& 0 \cr & \vdots & \end{matrix}\right), M_2 :=  \left(
\begin{matrix} u & 0 & 0 &0 \cr 0 & u & 0 & -x \cr 0 & 0 & 0 &
y\end{matrix} \right), \text{ and } M_1 := \left(\begin{matrix} w
\cr x \cr y \cr u
\end{matrix} \right). \]
(The algebra $A$ and the resolution above first appeared in
\cite{cspbw}.) Note that $A$ is quadratic but not Koszul, and
therefore is not $\ktwo$.

Let $\sigma = \id_A$ and define a derivation $\delta: A(-1)
\rightarrow A$ via $\delta(w) = w^2$, $\delta(x) = x^2$, $\delta(y)
= y^2$, and $\delta(u) = u^2$. Let $B = A[z; \delta]$ be the
associated Ore extension. The algebra $B$ is also quadratic.

We exhibit the chain map constructed by Lemma \ref{l:fexists}:
\begin{align*}
f_3: \bigoplus_{i \geq 4} B(-i) \xrightarrow{\left( \begin{matrix}
-z-y & 0 & 0 & 0 & \cr 0 & -z-y & 1 & 0 & \cdots \cr 0 & 0 & -z-y &
2 & \cr & & \vdots & &
\end{matrix} \right)}
 \bigoplus_{i \geq 3} B(-i), \cr
f_2: B(-3)^3 \xrightarrow{\left(\begin{matrix} z+u & 0 & 0 \cr 0 &
 z+ u+ x & 0 \cr 0 & 0 & -z-y\end{matrix}\right)} B(-2)^3, \cr
f_1: B(-2)^4 \xrightarrow{\left(\begin{matrix} -z-w & 0 & 0 &
 0 \cr 0 & -z-x & 0 & 0 \cr 0 & 0 & -z-y & 0 \cr 0 & 0 & 0 & -z-u
 \end{matrix}\right)} B(-1)^4, \cr
\text{ and } f_0: B(-1) \xrightarrow{\left(\begin{matrix} z
\end{matrix}\right)} B.
 \end{align*}
Note the constant entries in the matrix for $f_3$. (To see that this
is the correct matrix, use induction to show that $x^nz = zx^n +
nx^{n+1}$ in $B$, and therefore $yx^n(z+u) = (z+y) yx^n + nyx^{n+1}$
in $B$.) Now, the projective resolution $B \otimes W_\bullet$ for
$_B\bbk$ is given by
\begin{multline*}
0 \rightarrow B(-4, -5, -6, \dots) \xrightarrow{\left(\begin{matrix}
f_3 & M_3 \end{matrix}\right)} B(-3, -4, -5, \dots,) \oplus B(-3)^3
\xrightarrow{\left(\begin{matrix} M_3 & 0 \cr f_2 &
M_2\end{matrix}\right)} \cr B(-2)^3 \oplus B(-2)^4
\xrightarrow{\left(\begin{matrix} M_2 & 0 \cr f_1 & M_1
\end{matrix}\right)} B(-1)^4 \oplus B(-1)
\xrightarrow{\left(\begin{matrix} M_1 \cr f_0\end{matrix}\right)} B
\rightarrow \bbk \rightarrow 0.\end{multline*}

To calculate $\E(B)$, we apply $\Hom_B(-, \bbk)$ and use the natural
isomorphism $\Hom_B(B \otimes -, \bbk) \simeq \Hom_\bbk(-, \bbk)$ to
get the sequence \begin{multline*} 0 \rightarrow \Hom_\bbk(\bbk,
\bbk) \xrightarrow{0} \Hom_\bbk(\bbk(-1)^5, \bbk) \xrightarrow{0}
\Hom_\bbk(\bbk(-2)^7, \bbk) \xrightarrow{0} \cr
\Hom_\bbk\left(\bigoplus_{i \geq 3} \bbk(-i) \oplus \bbk(-3)^3,
\bbk\right) \xrightarrow{d} \Hom_\bbk\left(\bigoplus_{i \geq 4}
\bbk(-i), \bbk\right) \rightarrow 0.\end{multline*} For $j \geq 3$,
let $\lambda_j \in \Hom_\bbk\left(\bigoplus_{i \geq 3} \bbk(-i)
\oplus \bbk(-3)^3, \bbk\right)$ be the map $\bigoplus_{i \geq 3}
\bbk(-i) \rightarrow \bbk$ by projecting on the $(j-2)$th
coordinate. Similarly define $\rho_j \in\Hom_\bbk\left(\bigoplus_{i
\geq 4} \bbk(-i), \bbk\right)$ for $j \geq 4$. Then one can show
that for $j \geq 5$, $d(\lambda_j) = (j-3) \rho_j$, but that
$d(\lambda_3) = 0$ and $d(\lambda_4) = 0$. Hence $\E^{3,4}(B) \not =
0$, meaning $B$ is not Koszul (and therefore not $\ktwo$, since $B$
is quadratic).
\end{example}

\end{document}